\documentclass[letterpaper,10pt,conference]{ieeeconf}
\IEEEoverridecommandlockouts          
\usepackage{cite,mathtools,amssymb,amsfonts,mathrsfs,algorithm,algorithmic,graphicx,textcomp}
\usepackage[hidelinks]{hyperref}
\usepackage{pgfplots}
\usepackage{pgfplotstable}
\usepgfplotslibrary{groupplots,statistics}
\pgfplotsset{compat=1.18}
\usepackage{csvsimple}
\usepackage{siunitx}
\usepgfplotslibrary{groupplots}
\usetikzlibrary{calc,intersections}
\usepgfplotslibrary{groupplots,fillbetween}
\pgfplotsset{compat=newest}

\makeatletter
\providecommand\IEEEkeywordsname{Index Terms---}
\newenvironment{IEEEkeywords}
 {\par\addvspace{1.5ex}\noindent\footnotesize\IEEEkeywordsname\ignorespaces}
 {\par\addvspace{1.5ex}}
\makeatother

\newtheorem{proposition}{Proposition}
\newtheorem{assumption}{Assumption}
\newtheorem{definition}{Definition}

\newtheorem{remark}{Remark}

\newcounter{experiment}[section]

\usepackage{lettrine}  
\newcommand{\OPE}[1]{{\color{black}#1}}
\newcommand{\RK}[1]{{\color{black}#1}}

\DeclareMathOperator{\argmin}{arg\,min}

\begin{document}

\title{A Dynamic Mode Decomposition Approach to Parameter Identification}
\author{
    Moad Abudia\textsuperscript{1},
    Opeyemi Owolabi\textsuperscript{2}, 
    and Rushikesh Kamalapurkar\textsuperscript{2}%
    \thanks{\textsuperscript{1}School of Mechanical and Aerospace Engineering,
        Oklahoma State University,
        Stillwater, OK, 74074, USA
        (e‑mail: {\tt abudia@okstate.edu}).}%
    \thanks{\textsuperscript{2}Department of Mechanical and Aerospace Engineering, University of Florida,
        Gainesville, FL, 32611, USA (e‑mail: {\tt opeyemi.owolabi@ufl.edu}, {\tt rkamalapurkar@ufl.edu}).}%
       }
\maketitle

\thispagestyle{empty}
\pagestyle{empty}
\begin{abstract}
This paper \OPE{develops} a data-driven algorithm for simultaneous system \OPE{and parameter identification} in control-affine nonlinear systems.
Parameter \OPE{identification} is achieved by training a data-driven predictive model using \OPE{measurements collected across a prescribed set of known parameter values}. The predictive model is then used \OPE{to estimate unknown parameter values from measured trajectories by minimizing the trajectory prediction error using multi-start Nelder--Mead}.
Numerical experiments on the controlled Duffing oscillator with unknown damping, stiffness, and nonlinearity coefficients demonstrate accurate recovery of both the system trajectories and the unknown parameter values from data collected under open-loop excitation. \OPE{Additional studies evaluate parameter recovery across randomly selected parameter values and the effects of training-set size, query-horizon length, measurement noise, and process noise.}
\end{abstract}

\section{Introduction} \label{sec:intro}

Parameter estimation from data is a mature field, with a variety of well-established techniques for estimation of process parameters and construction of models from input-output data \cite{SCC.Astrom.Eykhoff1971}. This groundwork was subsequently formalized through gray-box modeling, prediction-error techniques, and recursive estimation \cite[Ch.~7]{SCC.Ljung1999}. Simultaneously, adaptive control introduced online update laws for adjusting parameters in uncertain systems during regulation and tracking \cite[Chs.~1,~6--7]{SCC.Ioannou.Sun1996}. In nonlinear control, such adaptive parameter update laws had a particularly significant impact in structured settings such as robot manipulator control, where adaptive designs leverage known regressors for uncertain dynamics \cite{SCC.Slotine.Li1987}. 
Unlike traditional parameter estimation problems, where the system dynamics depend on the parameters in a known way (e.g., linearly via known regressors \cite{SCC.Ogri.Bell.ea2023} or nonlinearly through known neural network architectures \cite{SCC.Raissi.Perdikaris.ea2019} and basis functions \cite{SCC.Zhou.Wang.ea2021}), the approach developed here does not assume a known functional relationship between the parameters and the system dynamics beyond continuity.

Many practical systems involve unknown parameters whose coupling to the dynamics is difficult to express in an explicit form. For example, a drone carrying an unknown payload experiences changes in mass and inertia that alter its flight dynamics \cite{SCC.Dai.Lee.ea2014,SCC.Outeiro.Cardeira.ea2021}. Similarly, unknown payload properties in robotic systems can modify the system dynamics and closed-loop behavior in ways that are difficult to characterize a priori \cite{SCC.Renner.Wind.ea2020}. These considerations motivate a fully data-driven approach capable of identifying parameters directly from input-output data without requiring prior knowledge of how those parameters couple with the system dynamics.

The data-driven character of operator-theoretic methods makes them a natural candidate to solve this problem, as they construct representations of nonlinear dynamics directly from trajectory data without requiring an explicit parametric model. Although such methods have matured significantly for control design, parameter identification remains largely separate from this framework. Classical system identification \cite[Ch.~7]{SCC.Ljung1999}, Bayesian estimation \cite{SCC.Bisaillon.Robinson.ea2024, SCC.Tripura.Chakraborty2023}, and neural inverse-problem methods \cite{SCC.Raissi.Perdikaris.ea2019, SCC.Gedon.Wahlstrom.ea2021} provide tools for recovering unknown physical parameters from data, but treat parameter recovery as a stage separate from operator-theoretic representations of the closed-loop dynamics. This paper addresses that gap by developing an operator-theoretic method for simultaneous system and parameter identification in control-affine nonlinear systems, while preserving the fully data-driven character of the underlying framework.
 
Among recent operator-theoretic approaches, the Discrete Control Liouville Dynamic Mode Decomposition (DCLDMD) framework of \cite{SCC.Morrison.Abudia.ea2024} provides a particularly strong foundation for controlled nonlinear systems. Building on the operator-theoretic formulations \cite{SCC.Mauroy.Goncalves2020, SCC.Rosenfeld.Kamalapurkar.ea2022, SCC.Gonzalez.Abudia.ea2024, SCC.Abudia.Rosenfeld.ea2024}, DCLDMD constructs a \OPE{finite-rank proxy of the closed-loop discrete control Liouville operator} from data collected under arbitrary \OPE{potentially open-loop control inputs}. This \OPE{construction uses a composition-like kernel propagation operator together with a multiplication operator that accounts for a feedback law}. \OPE{For the nonlinear kernels considered here, we instead use an empirical finite-data formulation of the operator.}

\OPE{The data-driven RKHS structure of DCLDMD motivates its extension to parameter identification} when unknown physical parameters are viewed as latent components of the dynamics. Motivated by this observation, this paper extends DCLDMD to simultaneous system and parameter identification, yielding the \OPE{proposed Discrete Control Liouville Identification (DCL-ID) framework}. The key idea is to augment the system state with the unknown physical parameters and \OPE{construct an empirical finite-data operator representation of the resulting augmented dynamics in the RKHS and vector-valued RKHS framework} \cite{SCC.Carmeli.DeVito.ea2010}. The resulting model can predict system trajectories corresponding to \OPE{candidate parameter values} starting from new initial conditions within the domain of validity of the model. This predictive capability is exploited to identify \OPE{the unknown parameter vector by minimizing the trajectory prediction error over a prescribed parameter set using multi-start Nelder--Mead \cite{neldermead1965}.} \OPE{Numerical experiments on a controlled Duffing oscillator evaluate parameter recovery over 30 independent parameter cases and study the effects of training-set size, query-\OPE{horizon} length, measurement noise, and process noise.}

This paper is organized as follows. \OPE{Section \ref{sec:ProbS} presents the problem description and the proposed approach.} \OPE{Section \ref{sec:Background} provides the mathematical background on reproducing kernel Hilbert spaces and vector-valued reproducing kernel Hilbert spaces.} \OPE{Sections \ref{sec:Finite-rank Representation}--\ref{sec:Pseudo-SVD} develop the empirical control Liouville operator and the associated pseudo-singular functions. Section \ref{sec:DCLDMDPID} presents the DCL-ID algorithm.} \OPE{Section \ref{sec:NumExp} presents numerical experiments on the controlled Duffing oscillator with unknown parameters.} \OPE{Section \ref{sec:discussion} discusses the numerical results, and Section X concludes the paper.}

\section{Overview} \label{sec:ProbS}

\subsection{Problem Statement}

Consider a nonlinear control-affine discrete-time dynamical system of the form
\[
z_{k+1}=f_0(z_k,\theta) + g_0(z_k,\theta)u_k,
\]
where \OPE{$z_k\in\mathbb{R}^n$ is the state, $u_k\in\mathbb{R}^m$ is the control input,} \(f_0\) and \(g_0\) are unknown continuous functions\OPE{,} and \(\theta\OPE{\in\mathbb{R}^p}\) denotes a vector of \OPE{physical} parameters. Unlike traditional \OPE{parameter-estimation} problems, where the \OPE{system dynamics depend on the parameters through a known functional form}, we do not make any assumptions on the \OPE{dependence of $f_0$ and $g_0$ on $\theta$} beyond continuity.

The objective in this manuscript is to \OPE{estimate an unknown but constant realization $\pi\in\mathbb{R}^p$ of the physical parameter vector $\theta$} \OPE{corresponding to} an \OPE{observed} trajectory of the system, i.e., measurements of the system state and the input expressed in the form of a \emph{query dataset} $\{\OPE{(z_k,u_k,w_k)}\}_{\OPE{k=0}}^{\OPE{N-1}}$, where
\begin{equation}
\label{eq:working_dataset}
w_{\OPE{k}}\coloneqq z_{\OPE{k+1}} = F(z_{\OPE{k}},u_{\OPE{k}},\pi) = f_0(z_{\OPE{k}},\pi) + g_0(z_{\OPE{k}},\pi) u_{\OPE{k}}.
\end{equation}
For example, in the case of a package delivery drone, the parameter of interest $\theta$ could correspond to \OPE{the} mass of a package, while $\pi$ denotes the \OPE{unknown mass associated with the query experiment}. The objective is then to estimate $\pi$.

\subsection{Approach}

In this paper, we solve the problem above by training a model that encodes the structural relationship between the parameters and the dynamical system. To train such a model, we need a \emph{training dataset} comprised of \OPE{labeled samples}
\[
\OPE{\mathcal{D}_{\mathrm{train}}\coloneqq\{(z_i,u_i,w_i,\pi_i)\}_{i=1}^{M}},
\]
where \OPE{$\pi_i$ is the known parameter vector associated with the $i$th sample}, satisfying
\begin{equation}
\label{eq:training_data}
\OPE{w_i=F(z_i,u_i,\pi_i)=f_0(z_i,\pi_i)+g_0(z_i,\pi_i)u_i.}
\end{equation}

We then train a model \OPE{$\hat{F}:\mathbb{R}^{n_a}\times\mathbb{R}^m\to\mathbb{R}^{n_a}$, where $n_a\coloneqq n+p$,} using the training dataset\OPE{. The physical-state component of $\hat{F}$ approximates $F$ and is used to infer an unknown parameter vector from the query dataset described in the previous section.}

In order to learn the parameters $\pi$, we need to solve \OPE{a trajectory-matching} optimization problem. Let $\Pi \subset \mathbb{R}^p$ be a known set containing $\pi$\OPE{, and let $\eta\in\Pi$ denote a candidate parameter vector. Define $x_0(\eta)=[z_0^\top,\eta^\top]^\top$ and recursively generate the corresponding ideal trajectory according to}
\begin{equation}
\label{eq:ideal_candidate_rollout}
\OPE{
x_{k+1}(\eta)=
\begin{bmatrix}
F(z_k(\eta),u_k,\eta)\\
\eta
\end{bmatrix},
}
\end{equation}
\OPE{where $z_k(\eta)$ denotes the physical-state component of $x_k(\eta)$.} The \OPE{ideal} optimization problem is
\begin{equation}
\label{eq:parameters-diffrence-minimization}
\hat{\pi} \OPE{\in} \argmin_{\eta\in \Pi} \frac{1}{N}\sum_{\OPE{k}=1}^{N}
\OPE{\left\|z_k-z_k(\eta)\right\|_2^2}.
\end{equation}

Since $F$ is unknown\OPE{, candidate trajectories are instead generated recursively using the learned model}. \OPE{Define the physical-state projection
$P_z\coloneqq\begin{bmatrix}I_n & 0_{n\times p}\end{bmatrix}$. For each candidate $\eta\in\Pi$, set $\hat{z}_0(\eta)=z_0$ and recursively define}
\begin{equation}
\label{eq:learned_candidate_rollout}
\OPE{
\hat{z}_{k+1}(\eta)=P_z\hat{F}\!\left(
\begin{bmatrix}
\hat{z}_k(\eta)\\
\eta
\end{bmatrix},u_k
\right).
}
\end{equation}
\OPE{The candidate parameter $\eta$ is held constant throughout each predicted trajectory, consistent with the assumption that the physical parameters remain constant over the observed trajectory.} The data-driven optimization problem is \OPE{therefore formulated as}
\begin{equation}
\label{eq:parameters-diffrence-minimization-with-sysID}
\hat{\pi} \OPE{\in} \argmin_{\eta\in \Pi} \OPE{J(\eta)}, \qquad
\OPE{
J(\eta)\coloneqq\frac{1}{N}\sum_{k=1}^{N}
\left\|z_k-\hat{z}_k(\eta)\right\|_2^2.
}
\end{equation}
\begin{assumption}\label{ass:parameter_optimization}
\OPE{The optimization problems in \eqref{eq:parameters-diffrence-minimization} and
\eqref{eq:parameters-diffrence-minimization-with-sysID}
admit at least one minimizer over $\Pi$. No uniqueness assumption is imposed.}
\end{assumption}

\OPE{The optimization problem in \eqref{eq:parameters-diffrence-minimization-with-sysID} is addressed numerically using multi-start Nelder--Mead \cite{neldermead1965} over the admissible parameter set $\Pi$, and the candidate yielding the smallest final value of $J$ is retained.}

In this paper, we adapt the DCLDMD technique developed in \cite{SCC.Morrison.Abudia.ea2024} to generate the model $\hat F$. In particular, \cite{SCC.Morrison.Abudia.ea2024} considered modeling of closed-loop dynamical systems under feedback control laws, whereas in this paper, we \OPE{use an empirical finite-data formulation based on DCLDMD to model the open-loop augmented dynamics}.

\section{Background} \label{sec:Background}
In this section, we provide a brief overview of reproducing kernel Hilbert spaces (RKHSs) and vector-valued reproducing kernel Hilbert spaces (vvRKHSs), and their role in \OPE{DCL-ID}.

\begin{definition}\label{Def:RKHS}
An RKHS $\Tilde{H}$ over a set $X\subset \mathbb{R}^{n}$ is a Hilbert space of functions $f: X\to \mathbb{C}$ such that for all $x\in X$ the evaluation functional $E_{x}f \coloneqq f(x)$ is bounded. By the Riesz representation theorem, there exists a function $\Tilde{K}_{x} \in \Tilde{H}$ such that $f(x) ={\langle f, \Tilde{K}_{x} \rangle}_{\Tilde{H}}$ for all $f\in \Tilde{H}$. 
\end{definition}

The snapshots of a dynamical system are embedded into an RKHS via a kernel map $x \mapsto \Tilde{K}(\cdot,x) \coloneqq \Tilde{K}_{x}$. Moreover, the span of the set $\{\Tilde{K}_{x} : x \in X\}$ is dense in $\Tilde{H}$.

\begin{proposition}\label{Prop:Kdensity}
If $\OPE{\mathcal{K}}\coloneqq\{\widetilde{K}_{x}:x\in X\}$, then $\OPE{\overline{\operatorname{span}\mathcal{K}}}=\widetilde{H}.$
\end{proposition}

\begin{proof}
To show that the span of the set $\{\widetilde{K}_{x}:x\in X\}$ is dense in $\widetilde{H}$, it suffices to show that $\OPE{\mathcal{K}^\perp}=\{0\}$. Let $h\in \OPE{\mathcal{K}^\perp}$. Then, for every $x\in X$, the reproducing property gives $\langle h,\widetilde{K}_{x}\rangle_{\widetilde{H}}=h(x)=0.$
Hence, $h\equiv0$ on $X$, and therefore $\OPE{\mathcal{K}^\perp}=\{0\}$. Thus, $\OPE{\overline{\operatorname{span}\mathcal{K}}=(\mathcal{K}^\perp)^\perp=\widetilde{H}},$ which proves the result.
\end{proof}
In order to account for the effect of control, we make use of a vvRKHS.

\begin{definition}\label{Def:vvRKHS}
Let $\mathcal{Y}$ be a Hilbert space, and let $H$ be a Hilbert space of functions from a set $X$ to $\mathcal{Y}$. The Hilbert space $H$ is a \emph{vvRKHS} if for every $\bar{u} \in \mathcal{Y}$ and $x \in X$, the functional $f \mapsto \langle f(x), \bar{u} \rangle_{\mathcal{Y}}$ is bounded.
\end{definition}

\OPE{For the augmented controlled system considered below, we specialize to $\mathcal{Y}=\mathbb{R}^{m+1}$ and $X\subseteq\mathbb{R}^{n_a}$, so $H$ is a vvRKHS of functions from $X$ to $\mathbb{R}^{m+1}$ and $\bar{u}=[1,u^\top]^\top\in\mathbb{R}^{m+1}$.}

\OPE{The vvRKHS $H$ is associated with an operator-valued reproducing kernel $K:X\times X\to\mathcal{L}(\mathcal{Y})$.} \OPE{For each $x\in X$ and $\bar{u}\in\mathcal{Y}$, define $K_{x,\bar{u}}\coloneqq K(\cdot,x)\bar{u}\in H$, following \cite{SCC.Rosenfeld.Kamalapurkar2021}.} The function $K_{x,\bar{u}}$ is known as the kernel operator and the span of these functions constitutes a dense set in the respective vvRKHS \cite[Proposition 1]{SCC.Rosenfeld.Kamalapurkar2021}. Given a function $f \in H$, the reproducing property of $K_{x,\bar{u}}$ implies ${\langle f, K_{x,\bar{u}} \rangle}_{H} = {\langle f(x), \bar{u} \rangle}_{\mathcal{Y}}$. For more discussion on vvRKHSs see \cite{SCC.Carmeli.DeVito.ea2010}.

\section{The \OPE{Empirical} Control Liouville Operator}
\label{sec:Finite-rank Representation}

In order to facilitate the description of the controlled dynamical system in terms of Koopman-like operators, the vvRKHS framework is utilized.
\OPE{For the augmented state $x=[z^\top,\pi^\top]^\top\in\mathbb{R}^{n_a}$, define
\[
f(x)\coloneqq
\begin{bmatrix}
f_0(z,\pi)\\
\pi
\end{bmatrix},
\qquad
g(x)\coloneqq
\begin{bmatrix}
g_0(z,\pi)\\
0_{p\times m}
\end{bmatrix},
\]
so that $x_{k+1}=f(x_k)+g(x_k)u_k$. Equivalently,
\[
F_a(x,u)\coloneqq f(x)+g(x)u
=
\begin{bmatrix}
F(z,u,\pi)\\
\pi
\end{bmatrix}.
\]
}\RK{Let $\tilde H_d$ be a RKHS with reproducing kernel $K_d$ and $H$ be the vvRKHS of functions from $X$ to $\mathbb{R}^{m+1}$ with reproducing kernel $K$. For $\{(x_i,u_i,y_i)\}_{i=1}^{M}$ with $y_i=F_a(x_i,u_i)$, let $\bar u_i\coloneqq[1,u_i^\top]^\top$, $\beta_i\coloneqq K_{x_i,\bar u_i}\in H$, $d_i\coloneqq K_d(\cdot,x_i)$, and $b_i\coloneqq K_d(\cdot,y_i)\in\tilde H_d$. Define $\beta\coloneqq[\beta_1,\ldots,\beta_M]^\top$,  $b\coloneqq[b_1,\ldots,b_M]^\top$, and $d\coloneqq[d_1,\ldots,d_M]^\top$ and the synthesis operators \OPE{$\Phi_M:\mathbb{R}^M\to H$, $\Delta_M:\mathbb{R}^M\to\tilde H_d$, and $\Psi_M:\mathbb{R}^M\to\tilde H_d$} by \OPE{$\Phi_M c\coloneqq\sum_{i=1}^{M}c_i\beta_i$, $\Delta_M c\coloneqq\sum_{i=1}^{M}c_i d_i$, and $\Psi_M c\coloneqq\sum_{i=1}^{M}c_i b_i$}. Let $P_d$, $P_b$, and $P_\beta$ denote the orthogonal projection onto $\operatorname{span}d$, $\operatorname{span}b$, and $\operatorname{span}\beta$, respectively.

The construction in \cite{SCC.Morrison.Abudia.ea2024} relies on the observation that if there exist operators $A^*:H\to \tilde{H}_d$ and $A = (A^*)^*$ such that $A^*K_{x,\overline{u}} = K_d(\cdot,y)$, for all $(x,y,u)$ satisfying $y = F_a(x,u)$, then for any observable $h\in \tilde{H}_d$,
\begin{equation*}
    \begin{aligned}
        \left([Ah](x)\right)\bar{u} &= \left\langle Ah,K_{x,\bar{u}}\right\rangle_{H} = \left\langle h,A^*K_{x,\bar{u}}\right\rangle_{\tilde{H}_d} \\
        &= \left\langle h,K_d(\cdot,y)\right\rangle_{\tilde{H}_d} = h(y)
    \end{aligned}
\end{equation*}
In particular, for the identity observable $(h_{\mathrm{id}})_j(x) = (x)_j$, where $(x)_j$ denotes the $j$th component of $x$,
\begin{equation*}
    \left([A(h_{\mathrm{id}})_j](x)\right)\overline{u} = (h_{\mathrm{id}})_j(y) = (y)_j,
\end{equation*}
which motivates a model of the form
\begin{equation*}
    x_{k+1} \approx \left([A(h_{\mathrm{id}})](x_k)\right)\overline{u}_k.
\end{equation*}
However, it is easy to see that for such an operator to exist,
$u\mapsto K_d(\bar x,F_a(x,u))$ must be affine in $u$. Since this affine-dependence condition is not generally satisfied by nonlinear kernels, we instead construct an empirical operator that satisfies $A^*K_{x_i,\overline{u}_i} = K_d(\cdot,y_i)$ only on the observed finite set of state-control transitions. We then show, through numerical experiments, that the approximation $x_{k+1} \approx \left([A(h_{\mathrm{id}})](x_k)\right)\overline{u}_k$ constructed using the empirical operator can effectively model nonlinear dynamical systems.
}

\begin{definition}
\OPE{For $\lambda=0$, define the empirical adjoint operator
\begin{equation}
\mathcal{L}_{M,0}^*:H\to\tilde H_d,
\qquad
\mathcal{L}_{M,0}^*
\coloneqq
\Psi_M G_\beta^+\Phi_M^*,
\label{eq:empirical_operator}
\end{equation}
where
\begin{equation}
G_\beta\coloneqq
\left(
\left\langle
K_{x_i,\bar u_i},K_{x_j,\bar u_j}
\right\rangle_H
\right)_{i,j=1}^{M}
=
\Phi_M^*\Phi_M
\label{eq:G_beta}
\end{equation}
is the Gram matrix associated with $\beta$.
For $\lambda>0$, define
\begin{equation}
\mathcal{L}_{M,\lambda}^*:H\to\tilde H_d,
\qquad
\mathcal{L}_{M,\lambda}^*
\coloneqq
\Psi_M(G_\beta+\lambda I_M)^{-1}\Phi_M^*,
\label{eq:regularized_empirical_operator}
\end{equation}
and let
$\mathcal{L}_{M,\lambda}\coloneqq(\mathcal{L}_{M,\lambda}^*)^*:
\tilde H_d\to H$.}
\end{definition}

\OPE{Each $\mathcal{L}_{M,\lambda}$ has rank at most $M$ and is therefore bounded.}

\begin{proposition}\label{prop:finite_data_interpolation}
\OPE{The empirical operator $\mathcal{L}_{M,0}^*$ satisfies
\[
\mathcal{L}_{M,0}^*\beta_i=b_i,\qquad i=1,\ldots,M,
\]
if and only if $\ker\Phi_M\subseteq\ker\Psi_M$.}
\end{proposition}

\begin{proof}
\OPE{$\mathcal{L}_{M,0}^*\Phi_M = \Psi_MG_\beta^+\Phi_M^*\Phi_M
= \Psi_MG_\beta^+G_\beta.$ Moreover, $G_\beta=\Phi_M^*\Phi_M$ implies $\ker G_\beta=\ker\Phi_M$, and since $\beta_i = \Phi_M e_i$,  $G_\beta^+G_\beta$ is the orthogonal projection onto $(\ker\Phi_M)^\perp$. Hence, $\mathcal{L}_{M,0}^*\Phi_M=\Psi_M$ if and only if $\Psi_M$ vanishes on $\ker\Phi_M$, which is equivalent to $\ker\Phi_M\subseteq\ker\Psi_M$. Finally, $\mathcal{L}_{M,0}^*\Phi_M=\Psi_M$ is equivalent to
$\mathcal{L}_{M,0}^*\beta_i=b_i$ for all $i$, which proves the result.}
\end{proof}

\OPE{In particular, if $G_\beta$ is nonsingular, then
\[
\mathcal{L}_{M,0}^*\beta_i
=
b_i
=
K_d(\cdot,y_i),
\qquad i=1,\ldots,M.
\]}

\begin{proposition}\label{prop:operator_regression}
\OPE{For $\lambda>0$, $\mathcal{L}_{M,\lambda}^*$ is the unique minimizer of
\begin{equation}
\min_{B\in\mathcal{S}_2(H,\tilde H_d)}
\left\{
\sum_{i=1}^{M}\|B\beta_i-b_i\|_{\tilde H_d}^2
+\lambda\|B\|_{\mathrm{HS}}^2
\right\},
\label{eq:operator_regression}
\end{equation}
where $\mathcal{S}_2(H,\tilde H_d)$ denotes the space of Hilbert--Schmidt operators from $H$ to $\tilde H_d$.}
\end{proposition}

\begin{proof}
\OPE{The objective in \eqref{eq:operator_regression} is $\|B\Phi_M-\Psi_M\|_{\mathrm{HS}}^2+\lambda\|B\|_{\mathrm{HS}}^2$. Its first-order optimality condition is $(B\Phi_M-\Psi_M)\Phi_M^*+\lambda B=0$, equivalently $B(\Phi_M\Phi_M^*+\lambda I_H)=\Psi_M\Phi_M^*$. Since $\lambda>0$, both $\Phi_M\Phi_M^*+\lambda I_H$ and $G_\beta+\lambda I_M$ are invertible. Moreover, $\Phi_M^*(\Phi_M\Phi_M^*+\lambda I_H)=(G_\beta+\lambda I_M)\Phi_M^*$, hence $\Phi_M^*(\Phi_M\Phi_M^*+\lambda I_H)^{-1}=(G_\beta+\lambda I_M)^{-1}\Phi_M^*$. Therefore,
\[
B=\Psi_M(G_\beta+\lambda I_M)^{-1}\Phi_M^*
=\mathcal{L}_{M,\lambda}^*.
\]
This operator is finite rank and therefore Hilbert--Schmidt. Since $\lambda>0$, the objective is strictly convex, so this stationary point is the unique minimizer.}
\end{proof}

\begin{definition}
Given \OPE{$\mathcal{L}_{M,0}|_{\operatorname{span}d}:
\operatorname{span}d\to\operatorname{span}\beta$}, the tuples
$\{(\sigma_i,\phi_i,\psi_i)\}_{i=1}^{\OPE{M}}$, with
$\sigma_i\in\OPE{\mathbb{R}_{\geq0}}$, $\phi_i\in\tilde H_d$, and
$\psi_i\in H$, are \OPE{pseudo-singular values, right pseudo-singular functions, and left pseudo-singular functions of $\mathcal{L}_{M,0}|_{\operatorname{span}d}$}, respectively, if \OPE{for all $h\in\operatorname{span}d$,}
\begin{equation}
\mathcal{L}_{M,0}h
=
\sum_{i=1}^{\OPE{M}}
\sigma_i\psi_i
\left\langle h,\phi_i\right\rangle_{\tilde H_d}.
\label{eq:infinite_spectral_reconstruction_svd}
\end{equation}
\end{definition}

\OPE{For $d=[d_1,\ldots,d_M]^\top$, define}
\begin{equation}
G_d\coloneqq
\left(
\left\langle
K_d(\cdot,x_i),K_d(\cdot,x_j)
\right\rangle_{\tilde H_d}
\right)_{i,j=1}^{M}
=
\Delta_M^*\Delta_M .
\label{eq:G_d}
\end{equation}

\OPE{Define the matrix}
\begin{equation}
I\coloneqq \left( \left\langle K_d(\cdot,x_i),K_d(\cdot,y_j)
\right\rangle_{\tilde H_d} \right)_{i,j=1}^{M}
= \Delta_M^*\Psi_M .
\label{eq:I_matrix}
\end{equation}

\OPE{The matrices $G_d$, $G_\beta$, and $I$ are used in the following section to obtain a \OPE{matrix} representation of $\mathcal{L}_{M,0}|_{\operatorname{span}d}$.}
\section{Matrix Representation of the Finite-rank Operator} \label{sec:Matrix Representation}

To formulate a matrix representation of the finite-rank operator \OPE{$\mathcal{L}_{M,0}$}, the operator is restricted to \OPE{$\operatorname{span}d$} to yield the operator \OPE{$\mathcal{L}_{M,0}|_{\operatorname{span}d}:\operatorname{span}d\to\operatorname{span}\beta$}.

\OPE{The inner products involving $K_d(\cdot,y_j)$ can be written as}
\begin{equation*}
\begin{aligned}
\OPE{\Psi_M^*h} &=
\OPE{ \begin{bmatrix}
\left\langle h,b_1\right\rangle_{\tilde H_d}\\
\vdots\\
\left\langle h,b_M\right\rangle_{\tilde H_d}
\end{bmatrix}} &=
\begin{bmatrix} \left\langle h,K_d(\cdot,y_1)\right\rangle_{\OPE{\tilde H_d}}\\
\vdots\\
\left\langle h,K_d(\cdot,y_M)\right\rangle_{\OPE{\tilde H_d}}
\end{bmatrix}.
\end{aligned}
\end{equation*}

\begin{proposition}
If $h=\OPE{c}^\top d\in\OPE{\operatorname{span}d}$ is a function with coefficients $\OPE{c}\in\mathbb{R}^M$ and \OPE{$r=\mathcal{L}_{M,0}h$}, then
\OPE{$r=a^\top\beta$}, where \OPE{$a=G_\beta^+I^\top\OPE{c}$} and $(\cdot)^+$ denotes the Moore--Penrose pseudoinverse.
\end{proposition}

\begin{proof}
Since \OPE{$\mathcal{L}_{M,0}=\Phi_MG_\beta^+\Psi_M^*$} and \OPE{$r=\mathcal{L}_{M,0}h$}, \OPE{$r=a^\top\beta$} with
\begin{multline}
\label{eq:alpha_projection}
a
=
\OPE{G_\beta^+\Psi_M^*h}
=
G_\beta^+
\begin{bmatrix}
\left\langle h,K_d(\cdot,y_1)\right\rangle_{\OPE{\tilde H_d}}\\
\vdots\\
\left\langle h,K_d(\cdot,y_M)\right\rangle_{\OPE{\tilde H_d}}
\end{bmatrix}.
\end{multline}
Using $h=\OPE{c}^\top d$,
\begin{equation}
\label{eq:beta_projection}
\OPE{\Psi_M^*h}=
\begin{bmatrix}
\left\langle \OPE{c}^\top d,K_d(\cdot,y_1)\right\rangle_{\OPE{\tilde H_d}}\\
\vdots\\
\left\langle \OPE{c}^\top d,K_d(\cdot,y_M)\right\rangle_{\OPE{\tilde H_d}}
\end{bmatrix}
=
\OPE{I^\top\OPE{c}},
\end{equation}
where \OPE{$I$ is defined in \eqref{eq:I_matrix}}. \OPE{Therefore,}
\begin{equation}
a=G_\beta^+I^\top\OPE{c}
\label{eq:b_matrix}
\end{equation}
\OPE{and} a matrix representation \OPE{$[\mathcal{L}_{M,0}]_d^\beta$} of the operator \OPE{$\mathcal{L}_{M,0}|_{\operatorname{span}d}$} is obtained as $G_\beta^+I^\top$.
\end{proof}

\OPE{For $\lambda>0$, the corresponding regularized matrix representation is
\begin{equation}
[\mathcal{L}_{M,\lambda}]_d^\beta
=
(G_\beta+\lambda I_M)^{-1}I^\top.
\label{eq:regularized_matrix_representation}
\end{equation}} Note that \OPE{when $G_\beta$ is singular, coefficient representations in $\operatorname{span}\beta$ need not be unique. The Moore--Penrose pseudoinverse in \eqref{eq:b_matrix} selects the minimum-norm coefficient vector.} In the case where the Gram matrix $G_\beta$ is nonsingular, \OPE{the coefficient representation is unique}, resulting in the unique matrix representation $G_\beta^{-1}I^\top$.

In the following section, the matrix representation \OPE{$[\mathcal{L}_{M,0}]_d^\beta$} is used to construct a data-driven representation of the \OPE{pseudo-singular values and the left and right pseudo-singular functions} of \OPE{$\mathcal{L}_{M,0}|_{\operatorname{span}d}$}.

\section{Pseudo-Singular Functions of the Finite-rank Operator} \label{sec:Pseudo-SVD}

\OPE{The} following proposition states that a decomposition of \OPE{$\mathcal{L}_{M,0}|_{\operatorname{span}d}$} analogous to a \OPE{singular value decomposition (SVD)} can be computed using matrices in the matrix representation \OPE{$[\mathcal{L}_{M,0}]_d^\beta$} derived in the previous section.

\begin{proposition}
If $(W,\Sigma,V)$ is the SVD of $G_\beta^{+}I^{\top}G_d^{+}$ with $W=\begin{bmatrix}w_1,&\ldots,&w_M\end{bmatrix}$, $V=\begin{bmatrix}v_1,&\ldots,&v_M\end{bmatrix}$, and $\Sigma=\mathrm{diag}\left(\begin{bmatrix}\sigma_1,&\ldots,&\sigma_M\end{bmatrix}\right)$, then for all $i=1,\ldots,M$, $\sigma_i$ are pseudo-singular values of \OPE{$\mathcal{L}_{M,0}|_{\operatorname{span}d}$} with \OPE{right} pseudo-singular functions $\phi_i\coloneqq v_i^\top d$ and \OPE{left} pseudo-singular functions $\psi_i\coloneqq w_i^\top\beta$.
\end{proposition}

\begin{proof}
Let $\phi_i=v_i^\top d$ and $\psi_i=w_i^\top\beta$ and $h=\OPE{c}^\top d$. Then,
\begin{multline*}
\OPE{\mathcal{L}_{M,0}}h=\sum_{i=1}^M\sigma_i\psi_i\left\langle h,\phi_i\right\rangle_{\tilde{H}_d}\\
\iff \OPE{\mathcal{L}_{M,0}}\OPE{c}^\top d=\sum_{i=1}^M\sigma_iw_i^\top\beta\left\langle\OPE{c}^\top d,v_i^\top d\right\rangle_{\tilde{H}_d}.
\end{multline*}
Using the finite-rank representation, the collection $\{(\sigma_i,\phi_i,\psi_i)\}_{i=1}^M$ \OPE{satisfies the SVD-like decomposition} of \OPE{$\mathcal{L}_{M,0}|_{\operatorname{span}d}$} if, for all $\OPE{c}\in\mathbb{R}^M$,
\begin{equation}\label{eq:suff_cond_SVD}
\left(G_\beta^{+}I^{\top}\OPE{c}\right)^\top\beta
=
\left(
\sum_{i=1}^M
\sigma_i
\left\langle\OPE{c}^\top d,v_i^\top d\right\rangle_{\tilde{H}_d}
w_i^\top
\right)\beta.
\end{equation}

\OPE{Since $G_d=\Delta_M^*\Delta_M$ from \eqref{eq:G_d}, we have $\ker G_d=\ker\Delta_M$. Also, since $I^\top=\Psi_M^*\Delta_M$ from \eqref{eq:I_matrix}, $\ker\Delta_M\subseteq\ker I^\top$, and therefore $I^\top=I^\top G_d^+G_d$.}

Simple matrix manipulations yield the chain of implications
\begin{gather*}
{\thickmuskip=0mu\thinmuskip=0mu\medmuskip=0mu
\eqref{eq:suff_cond_SVD}
\impliedby
\forall\OPE{c}\in\mathbb{R}^M,\;
G_\beta^{+}I^{\top}\OPE{c}
=
\sum_{i=1}^M
\sigma_i
\left\langle\OPE{c}^\top d,v_i^\top d\right\rangle_{\tilde{H}_d}
w_i}\\
\iff
\forall\OPE{c}\in\mathbb{R}^M,\;
G_\beta^{+}I^{\top}\OPE{c}
=
\sum_{i=1}^M
\sigma_i\left(w_iv_i^\top G_d\right)\OPE{c}\\
\impliedby
G_\beta^{+}I^{\top}
=
\sum_{i=1}^M
\sigma_i\left(w_iv_i^\top\right)G_d\\
\OPE{\impliedby}
G_\beta^{+}I^{\top}G_d^{+}
=
\sum_{i=1}^M\sigma_iw_iv_i^\top
=
W\Sigma V^\top,
\end{gather*}

This proves the proposition.
\end{proof}

\begin{remark}
The standard usage of the term SVD refers to a decomposition using orthonormal left and right singular vectors. The \OPE{right pseudo-singular functions} $\{\phi_i\}_{i=1}^{M}$ and \OPE{left pseudo-singular functions} $\{\psi_i\}_{i=1}^{M}$ are not necessarily orthonormal. Therefore, the decomposition in the previous proposition is not \OPE{a true }SVD of the finite-rank operator \OPE{$\mathcal{L}_{M,0}|_{\operatorname{span}d}$}.
\end{remark}
\section{The \OPE{DCL-ID} Algorithm} \label{sec:DCLDMDPID}
\subsection{DCL-ID Predictor}
In order to identify the parameters of the system
\[
z_{k+1}=f_0(z_k,\pi) + g_0(z_k,\pi)u_k,
\]
we first consider a modified system of the form
\begin{equation}
\label{eq:modified_system_with_parmeters_as_states}
y_k=x_{k+1}=
\begin{pmatrix}
z_{k+1}\\
\pi_{k+1}
\end{pmatrix}
=
\begin{pmatrix}
\OPE{f_0(z_k,\pi_k)}\\
\pi_k
\end{pmatrix}
+
\begin{pmatrix}
\OPE{g_0(z_k,\pi_k)}\\
0
\end{pmatrix}u_k,
\end{equation}
where $\pi_0=\pi$ \OPE{and $\pi_{k+1}=\pi_k$ for all $k$}. \OPE{Using the labelled training data in Section~\ref{sec:ProbS}, define $x_i=[z_i^\top,\pi_i^\top]^\top$ and $y_i=[w_i^\top,\pi_i^\top]^\top$, $i=1,\ldots,M$. The augmented samples satisfy $y_i=f(x_i)+g(x_i)u_i$.}

Motivated by \eqref{eq:infinite_spectral_reconstruction_svd}, assuming that \OPE{$(h_{\mathrm{id}})_j\in\tilde{H}_d$ for $j=1,\ldots,n_a$}, the system dynamics are approximated using the finite-rank operator \OPE{$\mathcal{L}_{M,0}$}, with rank at most $M$, as
\[
x_{k+1}\approx\hat{F}_M(x_k,u_k):=
\OPE{[\mathcal{L}_{M,0}P_dh_{\mathrm{id}}](x_k)}
\begin{pmatrix}
1\\
u_k
\end{pmatrix},
\]
where \OPE{$\mathcal{L}_{M,0}P_dh_{\mathrm{id}}$ denotes row-wise operation of $\mathcal{L}_{M,0}$ on the function $P_dh_{\mathrm{id}}$}.

Using the \OPE{pseudo-singular values and functions of $\mathcal{L}_{M,0}|_{\operatorname{span}d}$},
\begin{equation}
x_{k+1}\approx
\sum_{i=1}^M\sigma_i\xi_iw_i^\top\beta(x_k)
\begin{pmatrix}
1\\
u_k
\end{pmatrix}
=
\xi\Sigma W^\top\beta(x_k)
\begin{pmatrix}
1\\
u_k
\end{pmatrix},
\end{equation}
where $\xi_i\coloneqq\left\langle P_dh_{\mathrm{id}},\phi_i\right\rangle_{\tilde{H}_d}$ and $\xi\coloneqq\begin{bmatrix}\xi_1,&\ldots,&\xi_M\end{bmatrix}$.

The modes $\xi$ can be computed using $\phi_i=v_i^\top d$ as
\begin{gather*}
\xi=
\begin{bmatrix}
\left\langle P_dh_{\mathrm{id},1},v_1^\top d\right\rangle_{\tilde{H}_d} & \ldots & \left\langle P_dh_{\mathrm{id},1},v_M^\top d\right\rangle_{\tilde{H}_d}\\
\vdots & \ddots & \vdots\\
\left\langle P_dh_{\mathrm{id},\OPE{n_a}},v_1^\top d\right\rangle_{\tilde{H}_d} & \ldots & \left\langle P_dh_{\mathrm{id},\OPE{n_a}},v_M^\top d\right\rangle_{\tilde{H}_d}
\end{bmatrix}\\
=
\begin{bmatrix}
\left\langle\OPE{c}_1^\top d,d_1\right\rangle_{\tilde{H}_d} & \ldots & \left\langle\OPE{c}_1^\top d,d_M\right\rangle_{\tilde{H}_d}\\
\vdots & \ddots & \vdots\\
\left\langle\OPE{c}_{\OPE{n_a}}^\top d,d_1\right\rangle_{\tilde{H}_d} & \ldots & \left\langle\OPE{c}_{\OPE{n_a}}^\top d,d_M\right\rangle_{\tilde{H}_d}
\end{bmatrix}V
=
\OPE{c}^\top G_dV,
\end{gather*}
where $\OPE{c}\coloneqq\begin{bmatrix}\OPE{c}_1,&\ldots,&\OPE{c}_{\OPE{n_a}}\end{bmatrix}$. Using the reproducing property of the reproducing kernel of $\tilde{H}_d$, the coefficients $\OPE{c}_i$ in the projection of $h_{\mathrm{id},i}$ onto $d$ satisfy
{\thinmuskip=0mu \thickmuskip=0mu \medmuskip=0mu
\begin{equation*}
G_d\OPE{c}_i=
\begin{bmatrix}
\left\langle(h_{\mathrm{id}})_i,d_1\right\rangle_{\tilde{H}_d}\\
\vdots\\
\left\langle(h_{\mathrm{id}})_i,d_M\right\rangle_{\tilde{H}_d}
\end{bmatrix}
=
\begin{bmatrix}
(x_1)_i\\
\vdots\\
(x_M)_i
\end{bmatrix}.
\end{equation*}}
\OPE{With $D=\begin{bmatrix}x_1&\cdots&x_M\end{bmatrix}\in\mathbb{R}^{n_a\times M}$,} it can be concluded that $\OPE{c}^\top G_d=D$. Finally, the modes $\xi$ are given by $\xi=DV$ and the estimated open-loop model is given by
\begin{multline}
\label{eq:convergent_closed_loop_model}
x_{k+1}\approx\hat{F}_M(x_k,u_k)
=
DV\Sigma W^\top\beta(x_k)
\begin{pmatrix}
1\\
u_k
\end{pmatrix}\\
=
D(G_\beta^{+}I^\top G_d^{+})^\top\beta(x_k)
\begin{pmatrix}
1\\
u_k
\end{pmatrix}.
\end{multline}
\OPE{The drift approximation $\hat f_M(x)$ is given by the first column of
$D(G_{\beta}^{+}I^\top G_d^{+})^\top\beta(x)$,
while the control-effectiveness approximation $\hat g_M(x)$ is given by its last $m$ columns.}

\OPE{For numerical implementation, the following regularized predictor is evaluated directly rather than computing the pseudo-singular values and functions:}
\begin{equation}
\label{eq:regularized_dclid_predictor}
\OPE{
\begin{aligned}
\hat F_\lambda(x,u)
&\coloneqq
[\mathcal{L}_{M,\lambda}R_{d,\lambda}h_{\mathrm{id}}](x)
\begin{pmatrix}
1\\
u
\end{pmatrix}\\
&=
D(G_d+\lambda I_M)^{-1}
I
(G_{\beta}+\lambda I_M)^{-1}
\beta(x)
\begin{pmatrix}
1\\
u
\end{pmatrix}.
\end{aligned}
}
\end{equation}

\OPE{Here, $\hat F_M$ denotes the unregularized finite-rank predictor, while $\hat F_\lambda$ denotes the regularized predictor used in the numerical implementation.}

\subsection{Parameter Identification}

\OPE{Given the regularized predictor $\hat F_\lambda$, the parameter-identification problem in \eqref{eq:parameters-diffrence-minimization-with-sysID} is solved approximately using multi-start Nelder--Mead \cite{neldermead1965}. Since Nelder--Mead is local, no global-convergence guarantee is claimed; the candidate with the smallest final value of $J$ is retained.}

\OPE{For $\Pi=\prod_{j=1}^{p}[\underline{\eta}_j,\overline{\eta}_j]$, Algorithm~\ref{alg:DCLDMDPID} uses the smooth map $T:\mathbb{R}^p\to\operatorname{int}(\Pi)$, with $T_j(q_j)\in(\underline{\eta}_j,\overline{\eta}_j)$ for finite $q_j$ and the bounds approached asymptotically. Optimization is performed over $q$ using $J(T(q))$, with $\eta=T(q)$ held fixed while the predicted trajectory in \eqref{eq:learned_candidate_rollout} is generated.}

\begin{algorithm}[h]
\caption{\label{alg:DCLDMDPID}The DCL-ID algorithm}
\begin{algorithmic}[1]
\renewcommand{\algorithmicrequire}{\textbf{Input:}}
\renewcommand{\algorithmicensure}{\textbf{Output:}}

\REQUIRE \OPE{Training dataset $\{(z_i,u_i,w_i,\pi_i)\}_{i=1}^{M}$; query dataset $\{(z_k,u_k,w_k)\}_{k=0}^{N-1}$; kernels $K_d$ and $K$; regularization parameter $\lambda>0$; $\Pi=\prod_{j=1}^{p}[\underline{\eta}_j,\overline{\eta}_j]$; and Nelder--Mead initializations $\{q_0^{(\ell)}\}_{\ell=1}^{N_s}$.}

\ENSURE \OPE{$\hat F_\lambda$ and $\hat\pi$.}

\STATE \OPE{$x_i\leftarrow[z_i^\top,\pi_i^\top]^\top$, $y_i\leftarrow[w_i^\top,\pi_i^\top]^\top$, $\bar u_i\leftarrow[1,u_i^\top]^\top$, $i=1,\ldots,M$}

\STATE $\beta\leftarrow[K_{x_i,\bar u_i}]_{i=1}^{M}$

\STATE $G_\beta\leftarrow\left(\left\langle K_{x_i,\bar u_i},K_{x_j,\bar u_j}\right\rangle_H\right)_{i,j=1}^{M}$

\STATE \OPE{$G_d\leftarrow\left(\left\langle K_d(\cdot,x_i),K_d(\cdot,x_j)\right\rangle_{\tilde H_d}\right)_{i,j=1}^{M}$}

\STATE \OPE{$I\leftarrow\left(\left\langle K_d(\cdot,x_i),K_d(\cdot,y_j)\right\rangle_{\tilde H_d}\right)_{i,j=1}^{M}$}

\STATE \OPE{$D\leftarrow[x_1,\ldots,x_M]\in\mathbb{R}^{n_a\times M}$}

\STATE \OPE{$C_\lambda\leftarrow D(G_d+\lambda I_M)^{-1}I(G_\beta+\lambda I_M)^{-1}$}

\STATE \OPE{$\hat F_\lambda(x,u)\leftarrow C_\lambda\beta(x)\begin{pmatrix}1\\u\end{pmatrix}$}

\STATE \OPE{Define $T_j(q_j)\coloneqq \underline{\eta}_j+
(\overline{\eta}_j-\underline{\eta}_j)\frac{1+\tanh(q_j/2)}{2}$, $j=1,\ldots,p$, and $T(q)\coloneqq[T_1(q_1),\ldots,T_p(q_p)]^\top$.}

\STATE \OPE{For $\ell=1,\ldots,N_s$, apply Nelder--Mead from $q_0^{(\ell)}$ to minimize $J(T(q))$ using the predicted trajectory in \eqref{eq:learned_candidate_rollout}, with predictor $\hat F_\lambda$, while keeping the candidate parameter fixed, and obtain $\hat\pi^{(\ell)}$.}

\STATE \OPE{Select $s^\star\in\displaystyle\arg\min_{s\in\{1,\ldots,N_s\}}J(\hat\pi^{(s)})$ and set $\hat\pi\leftarrow\hat\pi^{(s^\star)}$.}

\RETURN \OPE{$\hat F_\lambda$ and $\hat\pi$.}

\end{algorithmic}
\end{algorithm}

\section{Numerical Experiments}\label{sec:NumExp}

\OPE{We evaluate DCL-ID on the controlled Duffing oscillator.}
The controlled Duffing oscillator is a nonlinear dynamical system with state-space form
\begin{equation}\label{eqn: Duffing}
\begin{pmatrix}
\OPE{(\dot{z})_{1}}\\
\OPE{(\dot{z})_{2}}
\end{pmatrix}
=
\begin{pmatrix}
\OPE{(z)_{2}}\\
-\delta \OPE{(z)_{2}}-\beta \OPE{(z)_{1}}-\alpha \OPE{(z)_{1}^{3}}
\end{pmatrix}
+
\begin{pmatrix}
0\\
2+\sin(\OPE{(z)_{1}})
\end{pmatrix}u.
\end{equation}
\OPE{Here, $z=[(z)_1,(z)_2]^\top\in\mathbb{R}^2$ denotes the physical state, $u\in\mathbb{R}$ is the control input, and $\alpha,\beta,\delta\in\mathbb{R}$ are the unknown parameters.}
\OPE{We augment the state with the parameters}
\begin{equation}\label{eqn: Duffing_modified}
\begin{pmatrix}
\OPE{(\dot{z})_{1}}\\
\OPE{(\dot{z})_{2}}\\
\dot{\alpha}\\
\dot{\beta}\\
\dot{\delta}
\end{pmatrix}
=
\begin{pmatrix}
\OPE{(z)_{2}}\\
-\delta \OPE{(z)_{2}}-\beta \OPE{(z)_{1}}-\alpha \OPE{(z)_{1}^{3}}\\
0\\
0\\
0
\end{pmatrix}
+
\begin{pmatrix}
0\\
2+\sin(\OPE{(z)_{1}})\\
0\\
0\\
0
\end{pmatrix}u.
\end{equation}
\OPE{Let $x=[z^\top,\alpha,\beta,\delta]^\top\in\mathbb{R}^5$ denote the augmented state. We discretize \eqref{eqn: Duffing_modified} using forward Euler with sampling time $T_s=0.1$~s, yielding a discrete-time control-affine system of the form $x_{k+1}=f_d(x_k)+g_d(x_k)u_k$.}
\OPE{In the DCL-ID implementation, we use $M=N_{\mathrm{train}}\in\{500,1000,2000,5000,10000\}$ training samples drawn independently and uniformly from $[-3,3]^5$, with control inputs sampled uniformly from $[-2,2]$. \OPE{To evaluate the effect of the size of the training dataset on the effectiveness of DCL-ID,} 10 independent trials are performed for each $N_{\mathrm{train}}$. Unless otherwise stated, the experiments use $N_{\mathrm{train}}=10{,}000$. \OPE{For the trajectory comparison in Fig.~\ref{fig:duffing_query_trajectory}, the true parameter vector is $\pi=[0.83,-1.17,0.21]^\top$.}}

\OPE{The query dataset consists of $N=50$ state transitions generated from $z_0=[1.5,-0.5]^\top$ under the periodic open-loop input pattern $[0.8,0.2,0.6,-0.7,0.5]$, repeated or truncated as required by the query horizon $N$. The unknown parameter vector $\pi=[\alpha,\beta,\delta]^\top$ is constrained to $\Pi=[-3,3]^3$ and estimated using Nelder--Mead from 24 distinct starting points. The estimate with the smallest value of $J$ is retained.}

\OPE{We use the exponential dot-product kernel
$K_d(x,y)=\exp(x^\top y/\rho_d)$ for $\tilde H_d$, with $\rho_d=20$.
For the vvRKHS $H$, we use the diagonal operator-valued kernel $K(x,y)=\operatorname{diag}\!\left(
\exp(x^\top y/\rho_1),\ldots,
\exp(x^\top y/\rho_{m+1})
\right),$ with $\rho_j=20$ for $j=1,\ldots,m+1$.}

We set \OPE{$\lambda=10^{-6}$} for regularization of the Gram matrices \OPE{and solve the regularized linear systems in \eqref{eq:regularized_dclid_predictor} without explicitly forming matrix inverses}.

\OPE{The query root-mean-square error (RMSE) is defined as
$\operatorname{RMSE}_{\mathrm{query}}
\coloneqq
\left(\frac{1}{nN}\sum_{k=1}^{N}\|z_k-\hat z_k(\hat\pi)\|_2^2\right)^{1/2}.$

To evaluate the estimated parameter on data not used to compute $\hat\pi$, we generate a separate noise-free trajectory of $N_{\mathrm{val}}=50$ transitions ($5$~s) from $z_0^{\mathrm{val}}=[-1.2,0.7]^\top$ using the periodic input pattern $[-0.4,0.7,-0.2,0.5,-0.6,0.1]$. Using the same initial condition and input sequence, the predicted states $\hat z_k^{\mathrm{val}}(\hat\pi)$ are generated using the learned model and the estimated parameter $\hat\pi$. The validation RMSE is
\begin{equation}
\label{eq:validation_rmse}
\operatorname{RMSE}_{\mathrm{val}}
\coloneqq
\left(\frac{1}{nN_{\mathrm{val}}}\sum_{k=1}^{N_{\mathrm{val}}}\|z_k^{\mathrm{val}}-\hat z_k^{\mathrm{val}}(\hat\pi)\|_2^2\right)^{1/2}.
\end{equation}}

\OPE{We evaluate parameter recovery over 30 independent cases with $N_{\mathrm{train}}=10{,}000$.} For the $r$th identification case, let $\pi^{\star,(r)}$ and $\hat\pi^{(r)}$ denote the true and estimated parameter vectors, where $\pi^{\star,(r)}$ is sampled independently and uniformly from $[0.575,1.425]\times[-1.425,-0.575]\times[0.10625,0.74375]$. \OPE{To study the effect of the number of query transitions, we use $N\in\{10,20,30,50,75,100\}$ with 10 trials for each value of $N$;} trial 1 uses $z_0=[1.5,-0.5]^\top$, while trials $2$--$10$ use $z_0+0.012\xi$, $\xi\sim\mathcal{N}(0,I_2)$, with independent perturbations for each trial and each value of $N$.

\OPE{The effect of noise is evaluated with $N_{\mathrm{train}}=5000$ for $\gamma\in\{0,0.01,0.025,0.05,0.1\}$, where $\gamma$ sets the noise standard deviation relative to the standard deviation of the noise-free values of each state component. For component $j$, $\sigma_{\nu,j}=\gamma\,\operatorname{std}((z)_j)$ and $\nu_{k,j}\sim\mathcal{N}(0,\sigma_{\nu,j}^2)$ denotes the noise added at step $k$, with samples redrawn if $|\nu_{k,j}|>3\sigma_{\nu,j}$. Let $\nu_k=[\nu_{k,1},\ldots,\nu_{k,n}]^\top$. For measurement noise, $\nu_k$ is added to the physical-state measurements in both the training data and the data used for parameter estimation, with $\tilde z_k=z_k+\nu_k$. For process noise, $\nu_k$ is added to the training successor states and the system evolves as $z_{k+1}=F(z_k,u_k,\pi)+\nu_k$ during parameter estimation. Ten trials are performed for each nonzero $\gamma$, with one noise-free trial for each noise type. Parameter labels and inputs remain noise-free, and validation uses noise-free data.}
\begin{figure}[!t]
\centering
\begin{tikzpicture}
\begin{groupplot}[
    group style={
        group size=1 by 2,
        vertical sep=1.00cm
    },
    width=0.92\columnwidth,
    height=0.60\columnwidth,
    enlarge x limits=0,
    enlarge y limits=0.08,
    label style={font=\scriptsize},
    tick label style={font=\scriptsize},
    title style={font=\scriptsize},
    grid=major,
    legend style={
        draw=none,
        font=\tiny,
        cells={anchor=west}
    }
]


\nextgroupplot[
    title={(a) First state component $(z)_1$},
    title style={
        at={(0.5,1.06)},
        anchor=north
    },
    ylabel={$(z)_1$}
]

\addplot[
    blue,
    thick
]
table[
    x=t,
    y=true_z1,
    col sep=tab
]{ACC2027/data/duffing_query_trajectory.dat};

\addplot[
    red,
    dashed,
    thick
]
table[
    x=t,
    y=dclid_z1,
    col sep=tab
]{ACC2027/data/duffing_query_trajectory.dat};


\nextgroupplot[
    title={(b) Second state component $(z)_2$},
    title style={
        at={(0.5,1.06)},
        anchor=north
    },
    xlabel={Time [s]},
    ylabel={$(z)_2$},
    legend style={
        at={(0.5,-0.30)},
        anchor=north,
        legend columns=2,
        /tikz/every even column/.append style={column sep=0.35cm}
    }
]

\addplot[
    blue,
    thick
]
table[
    x=t,
    y=true_z2,
    col sep=tab
]{ACC2027/data/duffing_query_trajectory.dat};

\addplot[
    red,
    dashed,
    thick
]
table[
    x=t,
    y=dclid_z2,
    col sep=tab
]{ACC2027/data/duffing_query_trajectory.dat};

\legend{True,DCL-ID}

\end{groupplot}
\end{tikzpicture}
\caption{Comparison of the Duffing oscillator query trajectory $z_k$ with the recursively generated DCL-ID trajectory $\hat z_k(\hat\pi)$ obtained using the identified parameter vector $\hat\pi$, for $N_{\mathrm{train}}=10{,}000$. Panels (a) and (b) show the $(z)_1$ and $(z)_2$ components, respectively.}
\label{fig:duffing_query_trajectory}
\end{figure}

\begin{figure}[!t]
\centering
\begin{tikzpicture}[
    /pgfplots/parity/.style={
        width=0.34\columnwidth,
        height=0.34\columnwidth,
        scale only axis,
        axis equal image,
        grid=major,
        tick align=outside,
        label style={font=\scriptsize},
        tick label style={font=\tiny},
        title style={font=\scriptsize},
        ylabel shift=-6pt,   
        xlabel shift=-2pt
    },
    exact/.style={black, dashed, thick},
    alphamk/.style={only marks, mark=*,         mark size=1.8pt, blue},
    deltamk/.style={only marks, mark=triangle*, mark size=2.0pt, red},
    betamk/.style ={only marks, mark=square*,   mark size=1.8pt, green!50!black}
]

\begin{groupplot}[
    group style={
        group size=2 by 1,
        horizontal sep=1.2cm
    },
    parity
]

\nextgroupplot[
    title={(a)},
    xlabel={$\alpha^{\star,(r)}$},
    ylabel={$\hat{\alpha}^{(r)}$},
    xmin=0.56, xmax=1.42,
    ymin=0.56, ymax=1.42,
    xtick={0.6,1.0,1.4},
    ytick={0.6,1.0,1.4}
]
\addplot[exact] coordinates {(0.56,0.56) (1.42,1.42)};
\addplot[alphamk] table[x=TrueP1, y=EstP1, col sep=tab]
    {ACC2027/data/duffing_parameter_recovery_raw.dat};

\nextgroupplot[
    title={(b)},
    xlabel={$\delta^{\star,(r)}$},
    ylabel={$\hat{\delta}^{(r)}$},
    xmin=0.09, xmax=0.76,
    ymin=0.09, ymax=0.76,
    xtick={0.1,0.4,0.7},
    ytick={0.1,0.4,0.7}
]
\addplot[exact] coordinates {(0.09,0.09) (0.76,0.76)};
\addplot[deltamk] table[x=TrueP3, y=EstP3, col sep=tab]
    {ACC2027/data/duffing_parameter_recovery_raw.dat};

\end{groupplot}

\begin{axis}[
    parity,
    at={(group c1r1.south east)},
    anchor=north,
    xshift=0.6cm,     
    yshift=-1.3cm,    
    title={(c)},
    xlabel={$\beta^{\star,(r)}$},
    ylabel={$\hat{\beta}^{(r)}$},
    xmin=-1.44, xmax=-0.58,
    ymin=-1.44, ymax=-0.58,
    xtick={-1.4,-1.0,-0.6},
    ytick={-1.4,-1.0,-0.6},
    legend style={
        at={(0.5,-0.30)},
        anchor=north,
        legend columns=-1,   
        draw=none,
        font=\tiny,
        /tikz/every even column/.append style={column sep=0.3cm}
    }
]

\addlegendimage{exact}
\addlegendimage{alphamk}
\addlegendimage{deltamk}
\addlegendimage{betamk}

\addplot[exact, forget plot] coordinates {(-1.44,-1.44) (-0.58,-0.58)};
\addplot[betamk, forget plot] table[x=TrueP2, y=EstP2, col sep=tab]
    {ACC2027/data/duffing_parameter_recovery_raw.dat};

\legend{Exact recovery, $\hat{\alpha}^{(r)}$, $\hat{\delta}^{(r)}$, $\hat{\beta}^{(r)}$}

\end{axis}

\end{tikzpicture}
\caption{Parameter recovery over 30 independent Duffing identification cases
for $N_{\mathrm{train}}=10{,}000$. Panels (a), (b), and (c) compare the true Duffing
parameters \OPE{$\alpha^{\star,(r)}$, $\delta^{\star,(r)}$, and $\beta^{\star,(r)}$} with the corresponding DCL-ID estimates \OPE{$\hat{\alpha}^{(r)}$, $\hat{\delta}^{(r)}$, and $\hat{\beta}^{(r)}$}, respectively. The dashed line in each panel indicates exact recovery, \OPE{$\hat{\pi}^{(r)}_j=\pi^{\star,(r)}_j$}.}
\label{fig:duffing_parameter_recovery}
\end{figure}

\begin{figure}[!t]
\centering
\begin{tikzpicture}[
    /pgfplots/trialmk/.style={only marks, mark=*, mark size=1.0pt,
                    black!45, fill opacity=0.6, draw opacity=0.6},
    /pgfplots/medln/.style={black, thick, mark=*, mark size=1.8pt,
                  error bars/.cd, y dir=both, y explicit,
                  error bar style={thin, solid, black},
                  error mark options={rotate=90, mark size=2pt, thin}}
]

\begin{groupplot}[
    group style={
        group size=2 by 1,
        horizontal sep=1.6cm
    },
    width=0.34\columnwidth,
    height=0.34\columnwidth,
    scale only axis,
    xmode=log,
    ymode=log,
    xmin=400, xmax=12500,
    grid=major,
    tick align=outside,
    label style={font=\scriptsize},
    tick label style={font=\tiny},
    title style={font=\scriptsize},
    xlabel={$N_{\mathrm{train}}$},
    xlabel shift=-2pt,
    ylabel shift=-4pt
]

\nextgroupplot[
    title={(a) Parameter error},
       ylabel={$\|\hat{\pi}^{(r)}-\pi^{\star,(r)}\|_2/\|\pi^{\star,(r)}\|_2$},
    legend to name=trainsizeleg,
    legend style={
        legend columns=-1,
        draw=none,
        font=\tiny,
        /tikz/every even column/.append style={column sep=0.35cm}
    }
]

\addplot[trialmk] table[
    x expr=\thisrow{Ntrain}*(1 + 0.04*((\thisrow{Trial}-5.5)/5.5)),
    y=RelativeParameterError,
    col sep=tab
]{ACC2027/data/duffing_training_size_raw.dat};

\addplot[medln] table[
    x=Ntrain,
    y=MedianRelativeParameterError,
    y error minus expr=\thisrow{MedianRelativeParameterError}-\thisrow{Q1RelativeParameterError},
    y error plus expr=\thisrow{Q3RelativeParameterError}-\thisrow{MedianRelativeParameterError},
    col sep=tab
]{ACC2027/data/duffing_training_size_quartiles.dat};

\legend{Individual trials, Median with IQR}

\nextgroupplot[
    title={(b) Validation error},
    ylabel={Validation RMSE}
]

\addplot[trialmk] table[
    x expr=\thisrow{Ntrain}*(1 + 0.04*((\thisrow{Trial}-5.5)/5.5)),
    y=ValidationRMSE,
    col sep=tab
]{ACC2027/data/duffing_training_size_raw.dat};

\addplot[medln] table[
    x=Ntrain,
    y=MedianValidationRMSE,
    y error minus expr=\thisrow{MedianValidationRMSE}-\thisrow{Q1ValidationRMSE},
    y error plus expr=\thisrow{Q3ValidationRMSE}-\thisrow{MedianValidationRMSE},
    col sep=tab
]{ACC2027/data/duffing_training_size_quartiles.dat};

\end{groupplot}

\node[anchor=north, yshift=-0.85cm]
    at ($(group c1r1.south east)!0.5!(group c2r1.south west)$)
    {\pgfplotslegendfromname{trainsizeleg}};

\end{tikzpicture}
\caption{Effect of the training-set size on DCL-ID performance over 10 independent trials for each $N_{\mathrm{train}}$. Faint markers show the individual trials, while the solid curves show the median with interquartile-range error bars. Panel (a) shows the relative parameter-estimation error, \OPE{$\|\hat{\pi}^{(r)}-\pi^{\star,(r)}\|_2/\|\pi^{\star,(r)}\|_2$}, and panel (b) shows the \OPE{validation RMSE}. Both axes are logarithmic.}
\label{fig:duffing_training_size}
\end{figure}

\begin{figure}[!t]
\centering
\includegraphics[width=.99\columnwidth]{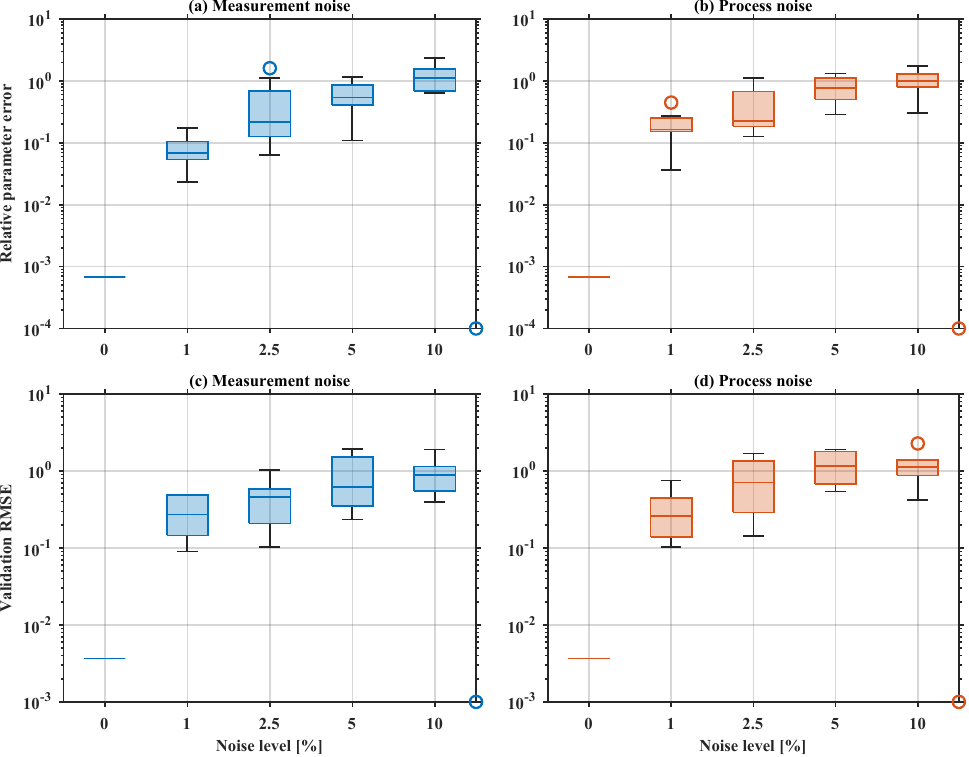}
\caption{Effect of measurement and process noise on DCL-ID for the Duffing oscillator. Box plots summarize the distributions over \OPE{10 trials at each nonzero noise level; the noise-free case is evaluated once for each noise type}. Panels (a) and (b) show the relative parameter error under measurement and process noise, respectively. Panels (c) and (d) show the corresponding validation RMSE. The vertical axes are logarithmic.}
\label{fig:duffing_noise_sensitivity}
\end{figure}

\begin{figure}[!t]
\centering
\begin{tikzpicture}[
    /pgfplots/trialmk/.style={
        only marks,
        mark=*,
        mark size=1.0pt,
        black!45,
        fill opacity=0.60,
        draw opacity=0.60,
        forget plot
    },
    /pgfplots/medln/.style={
        black,
        thick,
        mark=*,
        mark size=1.8pt,
        error bars/.cd,
        y dir=both,
        y explicit,
        error bar style={thin,solid,black},
        error mark options={rotate=90,mark size=2pt,thin}
    }
]

\begin{groupplot}[
    group style={
        group size=2 by 1,
        horizontal sep=1.4cm
    },
    width=0.34\columnwidth,
    height=0.34\columnwidth,
    scale only axis,
    ymajorgrids=true,
    xmajorgrids=false,
    grid style={gray!20},
    tick align=outside,
    xmin=5,
    xmax=105,
    xtick={10,20,30,50,75,100},
    xlabel={Number of query transitions $N$},
    label style={font=\scriptsize},
    tick label style={font=\tiny},
    title style={font=\scriptsize},
    xlabel shift=-2pt,
    ylabel shift=-4pt,
    scaled y ticks=false,
    yticklabel style={
        /pgf/number format/fixed,
        /pgf/number format/precision=1
    },
    unbounded coords=discard
]


\nextgroupplot[
    title={(a) Parameter error},
    ylabel={Rel.\ parameter error ($\times 10^{-3}$)},
    ymin=1.0,
    ymax=2.2,
    legend to name=horizonleg,
    legend columns=-1,
    legend style={
        draw=none,
        font=\tiny,
        /tikz/every even column/.append style={column sep=0.35cm}
    }
]

\addplot[trialmk]
table[
    x expr=\thisrow{Nquery}+0.30*(\thisrow{Trial}-5.5),
    y expr=1e3*\thisrow{RelativeParameterError},
    col sep=tab
]{ACC2027/data/duffing_query_horizon_raw.dat};

\addplot[medln]
table[
    x=Nquery,
    y expr=1e3*\thisrow{MedianRelativeParameterError},
    y error minus expr=
        1e3*(\thisrow{MedianRelativeParameterError}
        -\thisrow{Q1RelativeParameterError}),
    y error plus expr=
        1e3*(\thisrow{Q3RelativeParameterError}
        -\thisrow{MedianRelativeParameterError}),
    col sep=tab
]{ACC2027/data/duffing_query_horizon_quartiles.dat};

\addlegendimage{
    only marks,
    mark=*,
    mark size=1.0pt,
    black!45,
    fill opacity=0.60,
    draw opacity=0.60
}
\addlegendentry{Individual trials}

\addlegendimage{
    black,
    thick,
    mark=*,
    mark size=1.8pt
}
\addlegendentry{Median with IQR}


\nextgroupplot[
    title={(b) Validation error},
    ylabel={Validation RMSE ($\times 10^{-3}$)},
    ymin=1.0,
    ymax=3.0
]

\addplot[trialmk]
table[
    x expr=\thisrow{Nquery}+0.30*(\thisrow{Trial}-5.5),
    y expr=1e3*\thisrow{ValidationRMSE},
    col sep=tab
]{ACC2027/data/duffing_query_horizon_raw.dat};

\addplot[medln]
table[
    x=Nquery,
    y expr=1e3*\thisrow{MedianValidationRMSE},
    y error minus expr=
        1e3*(\thisrow{MedianValidationRMSE}
        -\thisrow{Q1ValidationRMSE}),
    y error plus expr=
        1e3*(\thisrow{Q3ValidationRMSE}
        -\thisrow{MedianValidationRMSE}),
    col sep=tab
]{ACC2027/data/duffing_query_horizon_quartiles.dat};

\end{groupplot}

\node[
    anchor=north,
    yshift=-0.85cm
]
at ($(group c1r1.south east)!0.5!(group c2r1.south west)$)
{\pgfplotslegendfromname{horizonleg}};

\end{tikzpicture}
\caption{Effect of the number of query transitions on DCL-ID over \OPE{10 trials for each value of $N$}. Faint markers show the individual trials, while the solid curves show the median with error bars spanning the interquartile range. Panel (a) shows the relative parameter error, and panel (b) shows the validation RMSE.}
\label{fig:duffing_query_horizon}
\end{figure}

\OPE{The reported identification time is the time required to run the 24 Nelder--Mead searches and evaluate $J(\eta)$ during those searches. The time required to construct $\hat F_\lambda$ and perform validation is not included.}

\begin{table}[t!]
\centering
\caption{\OPE{DCL-ID performance for the Duffing oscillator over 30 independent identification cases with $N_{\mathrm{train}}=10{,}000$. Results are reported as median [interquartile range (IQR)].}}
\label{tab:duffing_summary}
\renewcommand{\arraystretch}{1.10}
\setlength{\tabcolsep}{4pt}
\begin{tabular}{lc}
\hline
\textbf{Metric} & \textbf{Median [IQR]} \\
\hline
Parameter error, \OPE{$\|\hat{\pi}^{(r)}-\pi^{\star,(r)}\|_2$}
    & $2.42\times10^{-3}$ [$1.44\times10^{-3}$] \\
Relative parameter error
    & $1.53\times10^{-3}$ [$5.73\times10^{-4}$] \\
Query RMSE
    & $4.35\times10^{-5}$ [$3.48\times10^{-5}$] \\
Validation RMSE
    & $1.13\times10^{-3}$ [$2.18\times10^{-3}$] \\
Identification time [s]
    & $7.68$ [$9.53$] \\
\hline
\end{tabular}
\end{table}
\section{Discussion}
\label{sec:discussion}

\OPE{Figure~1 shows close agreement between the query trajectory and the DCL-ID predicted trajectory for $N_{\mathrm{train}}=10{,}000$. Figure~\ref{fig:duffing_parameter_recovery} shows that the estimated parameters remain close to the true values across 30 independent identification cases. Table~\ref{tab:duffing_summary} reports a median relative parameter error of $1.53\times10^{-3}$, with median query and validation RMSEs of $4.35\times10^{-5}$ and $1.13\times10^{-3}$, respectively. The median identification time is $7.68$~s.}

\OPE{Figure~\ref{fig:duffing_training_size} shows that parameter and validation errors generally decrease as the training-set size increases. The improvement becomes less pronounced for larger $N_{\mathrm{train}}$, and the validation error is not strictly monotonic.}

\OPE{Figure~\ref{fig:duffing_noise_sensitivity} shows generally larger parameter and validation errors as the measurement- and process-noise levels increase, although the trend is not monotonic at every noise level. The box plots also show increased trial-to-trial variation under noisy data.}

\OPE{Figure~\ref{fig:duffing_query_horizon} shows that the largest reduction in parameter and validation errors occurs as $N$ increases from 10 to 20. Both metrics change relatively little for larger values of $N$, indicating diminishing improvement as more query transitions are used.}

\section{Conclusion}

This paper presents DCL-ID, a data-driven approach for simultaneous system and parameter identification in control-affine nonlinear systems. An empirical finite-rank control Liouville operator is constructed from finite state-transition and input data, and a predictor of the augmented state dynamics is obtained from this operator. Unknown parameters are identified by minimizing the trajectory prediction error over a prescribed parameter set using multi-start Nelder--Mead.

Numerical experiments using the Duffing oscillator demonstrate accurate parameter recovery across independent identification cases and examine the effects of training-set size, the number of query transitions, measurement noise, and process noise. Future work will focus on reducing sensitivity to noisy data, improving computational efficiency for large training sets, incorporating partial system knowledge, and extending DCL-ID to higher-dimensional systems and validation using experimental data.
\section*{Acknowledgment}
\OPE{
The authors acknowledge UFIT Research Computing at the University of Florida for providing access to HiPerGator and computational resources that contributed to the numerical experiments reported in this paper.
}

\bibliographystyle{ieeeTRAN}
\bibliography{scc, sccmaster,DCLDMDbibliography,DCLDMDTemp,temp}

\end{document}